\documentclass[12pt,reqno]{amsart}
\usepackage{amssymb,amsmath,amsthm}
\usepackage{graphicx}
\usepackage{subcaption}
\usepackage{fullpage}
\usepackage{scrextend}
\usepackage{siunitx}
\usepackage{enumerate}
\usepackage{tikz,calc}
\usepackage{tikz-cd}
\usetikzlibrary{automata,positioning}
\usepackage{comment}
\usetikzlibrary{calc}
\usepackage{epsfig,graphicx}
\usepackage{listings}
\usepackage{enumitem}

\usepackage{float}
\usepackage{bbm}
\usepackage[colorlinks=true, pdfstartview=FitH, linkcolor=blue, citecolor=blue, urlcolor=blue]{hyperref}

\newcommand{\brax}[1]{\left( #1 \right)}
\newcommand{\abs}[1]{\left|#1\right|}

\newcommand{\C}{\mathbb{C}} 
\newcommand{\N}{\mathbb{N}}
 
\newcommand{\Z}{\mathbb{Z}}

\newcommand{\Mod}[1]{\ (\text{mod}\ #1)}
\renewcommand{\Mod}[1]{{\ifmmode\text{\rm\ (mod~$#1$)}\else\discretionary{}{}{\hbox{ }}\rm(mod~$#1$)\fi}}

\newtheorem{theorem}{Theorem}[section]
\newtheorem{prop}[theorem]{Proposition}
\newtheorem{lemma}[theorem]{Lemma}

\theoremstyle{definition} 

\newtheorem{remark}[theorem]{Remark}
\numberwithin{theorem}{section}

\begin{document}

\title{Improved estimates on the number of unit perimeter triangles }
\author{Ritesh Goenka}
\address{Department of Mathematics \\ University of British Columbia \\ Room 121, 1984 Mathematics Road \\ Vancouver, BC, Canada V6T 1Z2}
\email{rgoenka@math.ubc.ca}
\author{Kenneth Moore}
\address{Department of Mathematics \\ University of British Columbia \\ Room 121, 1984 Mathematics Road \\ Vancouver, BC, Canada V6T 1Z2}
\email{kjmoore@math.ubc.ca}
\author{Ethan Patrick White}
\address{Department of Mathematics \\ University of British Columbia \\ Room 121, 1984 Mathematics Road \\ Vancouver, BC, Canada V6T 1Z2}
\email{epwhite@math.ubc.ca}

\maketitle

\begin{abstract}
We obtain new upper and lower bounds on the number of unit perimeter triangles spanned by points in the plane. We also establish improved bounds in the special case where the point set is a section of the integer grid.
\end{abstract}

\section{Introduction}

A broad class of problems in extremal geometry can be characterized as follows. Fix a positive integer $k$ and a property of $k$-tuples of points in $\mathbb{R}^2$. The problem is then to determine how many $k$-tuples in an $n$-point set in $\mathbb{R}^2$ can have the chosen property. One of the most famous instances of this form of problem is the \emph{unit distance problem}, asked by Erd\H{o}s in 1946~\cite{ERD}. Here the special property is pairs of points being precisely distance one apart. Another well-studied open problem asked by Oppenheim in 1967 is to determine the maximum number of triples that determine a triangle of unit area~\cite[\S 6.2]{BMP}. The best known upper bounds for this problem were recently improved in~\cite{RS}.\\

In this work, we will focus on estimating the maximum number of unit perimeter triangles determined by $n$ points in the plane. This problem was first asked by Pach and Sharir, where they proved an upper bound of $O(n^{7/3})$~\cite{PaS2}. Later they improved their upper bound to $O(n^{16/7})$ and observed that a section of a triangular lattice spans $\Omega(n^{1+c/\log\log n})$ equilateral unit perimeter triangles~\cite{PaS}. In our main theorem, we report improvements to both estimates. \\

\begin{theorem}\label{mainT}
    The number of unit perimeter triangles spanned by $n$ points in the plane is $O(n^{9/4+\epsilon})$, for all $\epsilon>0$ with an implicit constant depending on $\epsilon$. Moreover, there exists a set of $n$ points in the plane that span $\Omega(n^{4/3})$ unit perimeter triangles. 
\end{theorem}

The upper bound given in Theorem~\ref{mainT} follows from an earlier argument of Pach and Sharir~\cite{PaS} combined with an improved incidence estimate of Sharir and Zahl~\cite{SZa}. Our lower bound construction uses a map of complex space to transfer a point-line incidence structure to a set of points such that a single point will belong to $\Omega(n^{4/3})$ unit perimeter triangles. We discuss the proof in Section~\ref{transformation}. \\

Grids and lattice structures are natural candidates for both the unit perimeter and unit area triangle problems. The best known construction in unit area triangle problem is a section of a scaled integer lattice~\cite{EP}. Raz, Sharir, and Shkredov improved estimates for the number of unit area triangles when the point set is of the form $A \times B$ where $A$ and $B$ are equal sized convex sets~\cite{RSS}. \\

We show that the number of equal perimeter triangles spanned by $[\lfloor \sqrt{n} \rfloor] \times [\lfloor \sqrt{n}\rfloor]$ is greater than the previous known lower bounds of this problem, but far fewer than the current best upper bound given in Theorem~\ref{mainT}. 

\begin{theorem}\label{gridT}
    The maximum number of equal perimeter triangles spanned by $[\lfloor \sqrt{n} \rfloor] \times [\lfloor \sqrt{n}\rfloor]$ is $O(n^{3/2+O\brax{\frac{1}{\log \log n}}})$ and $\Omega(n^{5/4})$.
\end{theorem}

The proof of Theorem~\ref{gridT} is discussed in Section~\ref{gridsec}. The upper bound follows from Proposition~\ref{prop:upper} and the lower bound follows from Proposition~\ref{prop:lower}. The key idea of the lower bound is that triangles of integer sidelengths and area can always be realized by the integer lattice.

\section{Proof of Theorem~\ref{mainT}}\label{transformation}

\noindent\textbf{Upper bound.} The idea behind Pach and Sharir's $O(n^{16/7})$ bound is the observation that for a fixed pair of points $a,b \in \mathbb{R}^2$, the set of points that form a unit perimeter triangle with $a$ and $b$ is an ellipse with foci $a,b$ and major axis length $1-|ab|$. Therefore, up to a multiplicative constant, the number of unit perimeter triangles form by $n$ points is equal to the number of incidences between $\binom{n}{2}$ ellipses and $n$ points. \\

Pach and Sharir bound these incidences using their own well-known incidence bound~\cite{PaS3}. We will use Sharir and Zahl's incidence bound, which relies on the following definition. An \emph{$s$-dimensional family of plane curves of degree at most $D$} is an algebraic variety $F \subset \textbf{P}\mathbb{R}^{\binom{D+2}{2}}$ that has dimension $s$. Theorem~1.3 of~\cite{SZa} states that the number of incidences between $m$ points and $n$ algebraic curves belonging to an $s$-dimensional family, no two sharing a common irreducible component is 
\begin{equation}\label{SZbound} O(m^{\frac{2s}{5s-4}}n^{\frac{5s-6}{5s-4}+\epsilon})+O(m^{2/3}n^{2/3} + m+n),\end{equation}
where $\epsilon>0$ is arbitrary, and the implicit constants above depend on $\epsilon$, and several parameters associated to the family of curves. \\

The set of all ellipses is a 5-dimensional family. The set of $\binom{n}{2}$ ellipses we construct all have the property that the sum of the distance between the foci and the major axis length is 1. This property can be expressed as an algebraic constraint on the coefficients of the ellipse, and so our ellipse set is a 4-dimensional family. Substituting values into \eqref{SZbound} gives the upper bound stated in Theorem~\ref{mainT}.\\

\noindent\textbf{Lower bound.} Our construction can be sketched in four steps as follows. 

\begin{enumerate}[label=(\arabic*),leftmargin=4\parindent]
    \item Construct a configuration of $n$ lines and $n$ points in $\mathbb{R}^2$ with $\Omega(n^{4/3})$ point-line incidences. By suitable translation, we will assume that no two lines (not necessarily distinct) are reflections of each other through the origin and all points are in the upper halfplane.
    \item Now identify $\mathbb{R}^2$ with $\mathbb{C}$ and apply the map $z \mapsto z^2$. The $n$ lines will be mapped to parabolas with focus at the origin. 
    \item Apply the map $z \mapsto \frac{z}{1+|z|}$. All parabolas from (2) are now mapped to ellipses with one focus at the origin, all tangent to the unit circle. Note that all triangles formed by the pair of foci of one such ellipse, and a point on the ellipse, have perimeter exactly 2. 
    \item At this stage we have $n$ points and $n$ ellipses with $\Omega(n^{4/3})$ point-ellipse incidences. We next enlarge our point set to include the origin, and the second foci of each ellipse. This is a total of at most $2n+1$ points, giving $\Omega(n^{4/3})$ triangles with perimeter exactly 2. 
\end{enumerate}

It remains to prove the claims in steps (2) and (3) above. We will do so in two lemmas.

\begin{lemma}
Let $f_1:\C \to \C$, be defined by $f_1(z)= z^2$. Then the image under $f_1$ of a line not through the origin is a parabola with a focus at the origin.
\label{lem:step2}
\end{lemma}

\begin{proof} Any line not containing the origin can be parameterized by 
\begin{equation}\label{lineeq}
 \ell(t) = re^{i\theta} \brax{1+it},
\end{equation}
where $re^{i\theta}$ is the point closest to the origin on the line. The image of $\ell(t)$ under the map $f_1$ is the curve 
\begin{equation*}
 \gamma_1(t)   = r^2e^{2i\theta} \brax{1-t^2 + 2 it}.
\end{equation*}

We will show that each point of $\gamma_1$ is equidistant to the origin and the line
\begin{equation}
\label{eq:directrix}
\begin{split}
\ell_D(t) = 2r^2e^{2i\theta}\brax{1 + it}.
\end{split}
\end{equation}

Fix any $t_0 \in \mathbb{R}$. The distance from $\gamma_1(t_0)$ to the origin is $r^2(1+t_0^2)$. On the other hand, we compute that 
\[ \abs{\ell_D(t)-\gamma_1(t_0)} = r^2\abs{ 1+t_0^2 + 2i(t-t_0)}, \]
which is minimized when $t=t_0$, matching the distance to the origin. Hence $\gamma_1$ is a parabola with directrix $\ell_D$ and focus at the origin. 
\end{proof}

After applying the map $z \mapsto z^2$ to $\mathbb{C}$ our incidence structure becomes a family of $n$ parabolas and $n$ points with $\Omega(n^{4/3})$ incidences. The condition that all points are in the upper halfplane guarantees the points remain distinct after applying $f_1$. Furthermore, since no pair of lines are a reflection of each other through the origin, all directrices of \eqref{eq:directrix} are also distinct.

\begin{lemma}
Let $f_2:\C\mapsto \C$, be defined by $f_2(z)= \frac{z}{1+\abs{z}}$. Then the image of a parabola with focus at the origin is an ellipse, with one focus at the origin, tangent to the unit circle.
\end{lemma}

\begin{proof}
From rotational symmetry, it is enough to verify the lemma for the parabola with directrix $\ell_D(t) = 2r^2(1+it)$, i.e., \eqref{eq:directrix} with $\theta = 0$. We will show the second focus is $q = -\frac{1}{1+r^2}$, and verify that the distances to the foci sum to a constant. The image of the parabola corresponding to $\ell_D$ is
\[ \gamma_2(t) = \frac{r^2 \brax{1+it}^2}{1+r^2(1+t^2)}. \]
A short computation shows that 
\begin{align*}
\abs{\gamma_2(t)-q} & = \frac{1}{(1+r^2(1+t^2))(1+r^2)}\abs{r^2(1+it)^2(1+r^2)+1+r^2(1+t^2)}\\
& = \frac{1}{(1+r^2(1+t^2))(1+r^2)}\abs{1+r^2+ir^2t}^2\\
& = \frac{(1+r^2)^2+r^4t^2}{(1+r^2(1+t^2))(1+r^2)}.
\end{align*}
Therefore the sum of the distances to our proposed foci is
\begin{align*}
    \abs{\gamma_2(t)-q} + \abs{\gamma_2(t)} & = \frac{(1+r^2)^2+r^4t^2}{(1+r^2(1+t^2))(1+r^2)} + \frac{r^2 \brax{1+t^2}}{1+r^2(1+t^2)}\\
    & = \frac{(1+2r^2)(1+r^2(1+t^2))}{(1+r^2(1+t^2))(1+r^2)} = \frac{1+2r^2}{1+r^2}.
\end{align*}

Since there is no dependence on $t$, the distance is constant, and therefore $\gamma_2$ describes an ellipse. It is clear that $\gamma_2$ lies entirely inside the unit disk, and that as $|t|\to\infty$, we have $\gamma_2(t) \to -1$, hence the tangency condition holds as well.
\end{proof}

\begin{figure}[h!]
    \centering
    \includegraphics[scale=0.5]{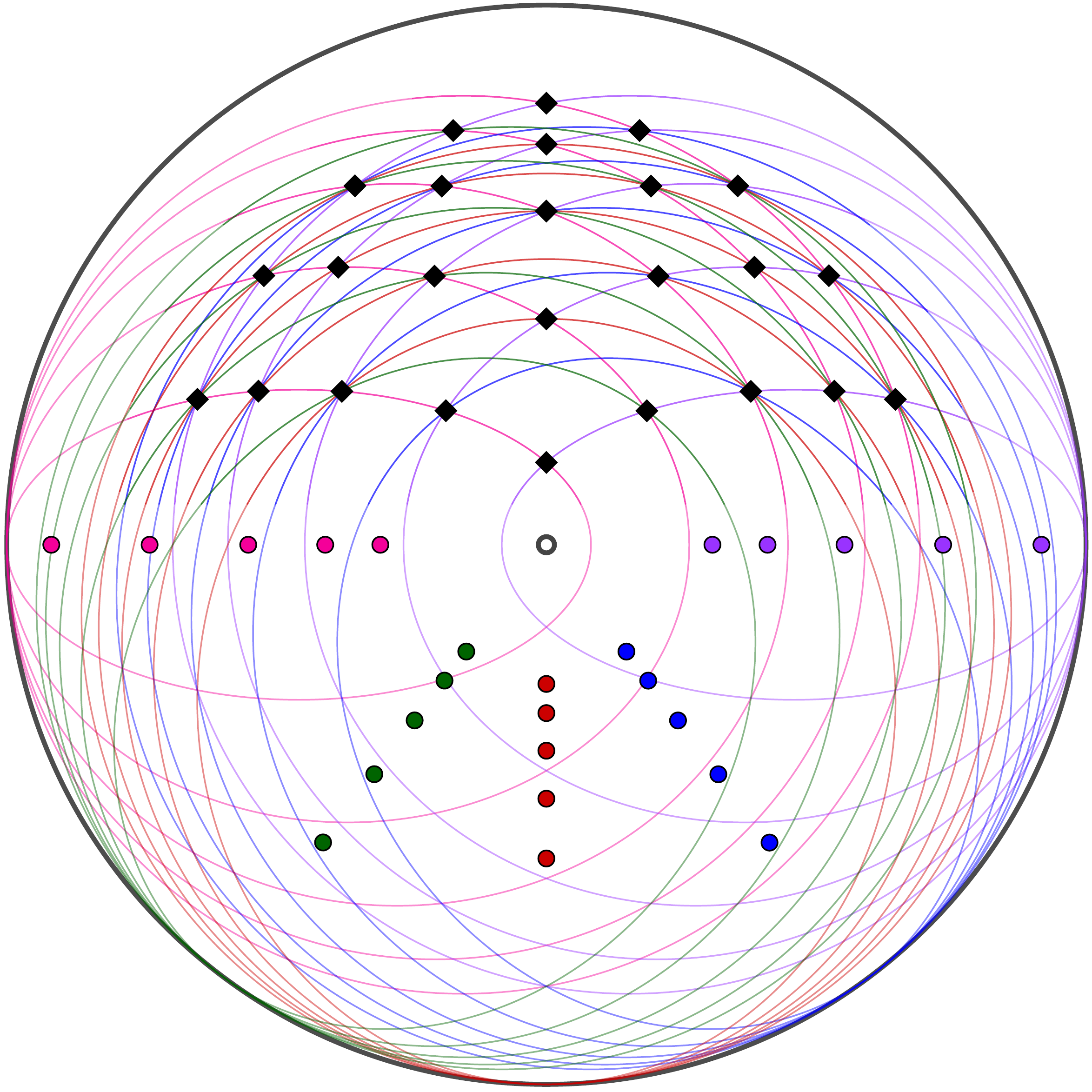}
    \caption{A $51$ point configuration forming $95$ unit perimeter triangles.}
    \label{fig:version3}
\end{figure}

An ellipse with one focus at the origin and tangent to the unit circle has the property that any triangle formed by the two foci and a point on the ellipse has perimeter exactly 2. Through this second map, we obtain $n$ ellipses and $n$ points with $\Omega(n^{4/3})$ point-curve incidences. Add to the point set the origin and each of the second foci of the ellipses, which are $n$ in number, for a total of at most $2n+1$ points in our set. This set determines $\Omega( n^{4/3})$ triangles with perimeter exactly 2. \\

In Figure~\ref{fig:version3}, we show an instance of this construction with $n=25$, and a total of $95$ unit perimeter triangles. The `square' points are the $f_2 \circ f_1$ image of a scaled $5 \times 5$ grid, the ellipses are the image of 5 familes of parallel lines, and the round points are the foci of the ellipses. 

\begin{remark}
    The map $z \mapsto \sqrt{z}$ can also yield interesting point-conic incidence configurations. In a similar fashion to the proof of Lemma \ref{lem:step2}, one can show that this map sends each line to one arc of a rectangular hyperbola.
\end{remark}

\section{Grid Construction}\label{gridsec}

We now turn to estimating the number of unit perimeter triangles in the grid $[\lfloor \sqrt{n} \rfloor] \times [\lfloor \sqrt{n}\rfloor]$. The following notation and terminology will be used. Let $\mathbf{d}(n)$ denote the number of positive divisors of $n$. Whenever we refer to a triangle, it is assumed to be non-degenerate. A triangle is a \emph{lattice triangle} if it can be realized with vertices in the integer lattice $\Z^2$. Further, a triangle is said to be \emph{Heronian} if its sidelengths and area are integers. For $n \in \N$, let $H(n)$ denote the number of non-congruent Heronian triangles with perimeter $n$.

\begin{lemma}
\label{prop:cong}
    For $p > 0$, the number of non-congruent lattice triangles with perimeter $p$ is $O\brax{p^{1+O\brax{\frac{1}{\log \log p}}}}$.
\end{lemma}

\begin{proof} The square distance of any side of a lattice triangle is a positive integer. We may assume that the perimeter $p$ is realized by at least one lattice triangle. Let $a_i\sqrt{d_i}$ for $i=1,2,3$ be the sidelengths of a triangle with perimeter $p$, where $a_i,d_i \in \mathbb{N}$ and $d_i$ are squarefree for $i=1,2,3$. Suppose $b_i\sqrt{e_i}$ are the sidelengths of a different triangle also with perimeter $p$. A folklore result shows that the set of all square roots of squarefree natural numbers is linearly independent over the rationals~\cite{bore,CaS}. Hence since
\begin{equation}\label{equalp} a_1\sqrt{d_1}+a_2\sqrt{d_2}+a_3\sqrt{d_3}=b_1\sqrt{e_1}+b_2\sqrt{e_2}+b_3\sqrt{e_3}, \end{equation}
we must have that $\{d_1,d_2,d_3\} = \{e_1,e_2,e_3\}$. Note that if $d_i$ for $i=1,2,3$ are distinct then \eqref{equalp} can only hold if $\{a_i\sqrt{d_i}\}\}_{i=1}^3 =\{b_i\sqrt{e_i}\}\}_{i=1}^3$, i.e., there is a unique triangle with perimeter $p$. Similarly, if $d_1 = d_2 \neq d_3$ then without loss of generality we may assume $d_1 = e_1 = e_2$ and $d_3 = e_3$. This implies $a_3 = b_3$ and $a_1+a_2 = b_1+b_2$. Since $a_1+a_2 < p$, the number of solutions $(b_1,b_2,b_3)$ to \eqref{equalp} is less than $p$. \\

Finally, we consider the most interesting case: $\{d_1\} =\{d_1,d_2,d_3\} = \{e_1,e_2,e_3\}$. Let $a_1+a_2+a_3 = b_1+b_2+b_3= m$. By Pick's Theorem~\cite{Pic}, the area of a lattice triangle is a half-integer. For some $k \in \mathbb{N}$, let $k/2$ be the area of the triangle with sidelengths $\{b_i\sqrt{e_i}\}_{i=1}^3$. By Heron's area formula we have 
\begin{equation}\label{heron}
    d_1^2m(m-2b_1)(m-2b_2)(m-2b_3) = 4k^2.
\end{equation}
From \eqref{heron}, we see that $d_1m$ is even. We will consider \eqref{heron} as a Diophantine equation, and upper bound the number of solutions for a fixed choice of $0<b_1<m/2$. Put $x = b_2-b_3$ and let $c = d_1^2m(m-2b_1)/4\in \mathbb{N}$. Equation \eqref{heron} becomes
\begin{equation}\label{he_dio}
 k^2+cx^2 = cb_1^2.
\end{equation}
Each integer solution $(k,x)$ to the above determines at most one unique triangle with perimeter $p$. From \cite[Theorem~2]{Ver} the number of integer solutions to \eqref{he_dio} is bounded above by $24 \textbf{d}(64b_1^2c^2)$. Since $b_1\sqrt{d_1}<p/2$ and $c < d_1p^2/4$ we conclude $64b_1^2c^2 < p^8$. From \cite[Theorem~1]{NiR} we have that 
\[ \log \textbf{d}(64b_1^2c^2) \leq \frac{2\log p^8}{\log \log p^8},  \]
from which we derive $24 \textbf{d}(64b_1^2c^2) = O(p^{16/\log\log p})$. Since there are less than $p$ choices for $b_1$, the stated bound follows.
\end{proof}

The following proposition gives the upper bound on the number of equal perimeter triangles in grids stated in Theorem~\ref{gridT}.

\begin{prop}\label{prop:upper}
    For $n \in \N$, the number of triangles with vertices in $[n] \times [n]$ having a fixed perimeter $p \in \mathbb{R}_{>0}$ is $O\brax{n^{3+O\brax{\frac{1}{\log \log n}}}}$.
\end{prop}

\begin{proof}
   Note that we may assume $p<4n$, since all triangles are contained in the square circumscribing the whole grid. From Lemma~\ref{prop:cong}, the number of non-congruent lattice triangles with perimeter $p$ is $O\brax{p^{1+O\brax{\frac{1}{\log \log p}}}} = O\brax{n^{1+O\brax{\frac{1}{\log \log n}}}}$. Let $\Delta = ABC$ be a triangle from one of these congruence classes. We have at most $|[n] \times [n]| = n^2$ choices for the point $A$. Let us fix a point $A = (x_A, y_A)$. Then the coordinates $(x, y)$ of point $B$ satisfy 
    \begin{equation*}
        (x - x_A)^2 + (y - y_A)^2 = |AB|^2 < 2n^2,
    \end{equation*}
    where $|AB|$ denotes the length of the side $AB$.
    By Jacobi's two squares theorem \cite[Sec.~9.1]{BoB} (see also \cite[Theorem~2]{Ver}), the number of integer solutions $(x, y)$ to the above equation is less than or equal to $4 \mathbf{d}(|AB|^2) = O\brax{n^{O\brax{\frac{1}{\log \log n}}}}$, yielding at most as many choices for the point $B$. For each fixed pair of points $A$ and $B$, there are at most two lattice points $C$ such that the triangle $ABC$ is congruent to $\Delta$. Hence, the total number of triangles with vertices in $[n] \times [n]$ and perimeter $p$ is bounded above by
    \begin{equation*}
        O\brax{n^{1+O\brax{\frac{1}{\log \log n}}}} \cdot n^2 \cdot O\brax{n^{O\brax{\frac{1}{\log \log n}}}} \cdot 2 = O\brax{n^{3+O\brax{\frac{1}{\log \log n}}}}.
    \end{equation*}
\end{proof}

We now show that there are $\Omega(n^{5/2})$ triangles of the same perimeter in $[n] \times [n]$. Our first step is to find a lower bound on the number of Heronian triangles. We expect the estimate of Lemma~\ref{prop:hero} can be improved, at least for some special perimeters. 

\begin{lemma}
\label{prop:hero}
    For $k \in \N$, we have $H(2(2k-1)^2) = \Omega(k)$. 
\end{lemma}

\begin{proof}
    Let $k \in \N$ and $p = 2(2k-1)^2$. Let $\Delta$ be a triangle with sides $a, b, c$, and perimeter $p$. The semi-perimeter $s$ of $\Delta$ is equal to $(2k-1)^2$. The condition that $\Delta$ is a Heronian triangle is equivalent to the condition that
    \begin{equation}
    \label{eqn:area}
        s(s-a)(s-b)(s-c)
    \end{equation}
    is a perfect square, and $a,b,c \in \N$ are all strictly smaller than $s$. Note that \eqref{eqn:area} is a perfect square if each of $(s-a)$, $(s-b)$, and $(s-c)$ is a perfect square. Further, note that \begin{equation*}
        (s-a) + (s-b) + (s-c) = 3s - (a+b+c) = 3s - 2s = s.
    \end{equation*}
    Therefore, any integral solution $(x,y,z)$ to the equation
    \begin{equation*}
        x^2 + y^2 + z^2 = s = (2k-1)^2,
    \end{equation*}
    with $xyz \ne 0$ corresponds to a Heronian triangle with sides $(s-x^2)$, $(s-y^2)$, and $(s-z^2)$. By accounting for signs and permutations of variables, we have the estimates
    \begin{align}
        3! \cdot H(p) &\ge |\{(x,y,z) \in \Z^3: x^2 + y^2 + z^2 = s \text{ and } x, y, z > 0 \}|, \nonumber \\
        &= \frac{1}{2^3} |\{(x,y,z) \in \Z^3: x^2 + y^2 + z^2 = s \text{ and } xyz \ne 0 \}|, \nonumber \\
        &\ge \frac{1}{2^3} \brax{|\{(x,y,z) \in \Z^3: x^2 + y^2 + z^2 = s\}| - 3 \cdot |\{(x,y) \in \Z^2: x^2 + y^2 = s\}|}
        \label{eqn:sphere}
    \end{align}
    Note that $s \equiv 1 \; (\text{mod } 8)$ since it is the square of an odd integer. It follows from an observation of Hurwitz, verified by Olds~\cite{Old} that
    \begin{equation}
    \label{eqn:sphere-est}
        |\{(x,y,z) \in \Z^3: x^2 + y^2 + z^2 = s\}| \ge 6 \sqrt{s} = \Omega(k).
    \end{equation}
    Furthermore, by Jacobi's two squares theorem \cite[Sec.~9.1]{BoB}, 
    \begin{equation}
    \label{eqn:circle-est}
        |\{(x,y) \in \Z^2: x^2 + y^2 = s\}| \le 4 \mathbf{d}(s) = O\brax{s^{O\brax{\frac{1}{\log \log s}}}} = O\brax{k^{O\brax{\frac{1}{\log \log k}}}}.
    \end{equation}
    Finally, using estimates \eqref{eqn:sphere-est} and \eqref{eqn:circle-est} in \eqref{eqn:sphere} yields the desired result.
\end{proof}

\begin{prop}\label{prop:lower}
    For $n \in \N$, there are $\Omega(n^{\frac{5}{2}})$ triangles with vertices in $[n] \times [n]$ having the same perimeter.
\end{prop}

\begin{proof}
    Suppose that $n \ge 400$. Set $k = \lfloor \sqrt{n}/20 \rfloor$ so that $p = 2(2k-1)^2 \le n/50$. From Lemma~\ref{prop:hero}, we have $H(p) = \Omega(k) = \Omega(\sqrt{n})$. Let $\Delta$ be a Heronian triangle with perimeter $p$. It follows from \cite[Theorem~1]{Yiu} that $\Delta$ is a lattice triangle. Further, any such realization can be enclosed in a lattice square having sides parallel to the coordinate axes with sidelength $p/2 \le n/100$. There are at least $(n - n/100)^2$ choices for the lower left vertex of the enclosing square. Therefore, the number of triangles with perimeter $p$ and vertices in $[n] \times [n]$ is bounded below by
    \begin{equation*}
        H(p) (n - n/100)^2 = \Omega(\sqrt{n}) \Omega(n^2) = \Omega(n^{\frac{5}{2}}).
    \end{equation*}
\end{proof}

\section*{Acknowledgments}

The authors would like to thank Michael Bennett, Jozsef Solymosi, and Joshua Zahl for helpful discussions during the preparation of this work.


\begin{thebibliography}{9}

\bibitem{bore} I. Boreico, My Favorite Problem: Linear Independence of Radicals, Harvard College Math Review, \textbf{2}(1) (2008) 87--93. 

\bibitem{BoB} J. M. Borwein, P. B. Borwein, Pi and the AGM, Wiley, New York, 1987. MR0877728

\bibitem{BMP} P. Brass, W. Moser, J. Pach, Research problems in discrete geometry, Springer, New York, 2005. MR2163782

\bibitem{CaS} R. Carr, C. O'Sullivan, On the linear independence of roots, Int. J. Number Theory \textbf{5}(1) (2009) 161--171. MR2499028

\bibitem{ERD} P. Erd\H{o}s, On sets of distances of $n$ points, Amer. Math. Monthly \textbf{53} (1946) 248--250. MR0015796

\bibitem{EP} P. Erd\H{o}s, G. Purdy, Some extremal problems in geometry, J. Combin. Theory Ser. A \textbf{10} (1971) 246--252. MR0275288

\bibitem{NiR} J.-L. Nicolas, G. Robin, Majorations explicites pour le nombre de diviseurs de $N$, Can. Math. Bull. \textbf{26}(4) (1983) 485--492. MR0716590

\bibitem{Old} C. D. Olds, On the representations, $N_3(n^2)$, Bull. Amer. Math. Soc. \textbf{47} (1941) 499--503. MR0004262

\bibitem{PaS2} J. Pach, M. Sharir, Repeated angles in the plane and related problems, J. Combin. Theory, Ser. A \textbf{59} (1992) 12--22. MR1141318

\bibitem{PaS3} J. Pach, M. Sharir, On the number of incidences between points and curves, Combin. Probab. Comput. \textbf{7} (1998) 121--127. MR1611057

\bibitem{PaS} J. Pach, M. Sharir, Geometric incidences, in Towards a Theory of Geometric Graphs, Contemp. Math. \textbf{342}, Amer. Math. Soc., Providence, RI, 2004, pp. 185--223. MR2065264

\bibitem{Pic} G. Pick, Geometrisches zur Zahlenlehre, Sitzungber. Lotos (Prague) \textbf{19} (1899) 311--319.

\bibitem{RS} O.E. Raz, M. Sharir, The number of unit-area triangles in the plane: theme and variations, Combinatorica \textbf{37} (2017) 1221--1240. MR3759915

\bibitem{RSS} O.E. Raz, M. Sharir, I.D. Shkredov, On the number of unit-area triangles spanned by convex grids in the plane, Comput. Geom. \textbf{62} (2017) 25--33. MR3603773

\bibitem{SZa} M. Sharir, J. Zahl, Cutting algebraic curves into pseudo-segments and applications, J. Comb. Theory, Ser. A \textbf{150} (2017) 1--35. MR3645567

\bibitem{Ver} V. Shelestunova, Upper bounds for the number of integral points on quadratic curves and surfaces, Ph. D. thesis, Univ. of Waterloo, Ontario, Canada, 2010.

\bibitem{Yiu} P. Yiu, Heronian triangles are lattice triangles, Am. Math. Mon. \textbf{108} (2001) 261--263. MR1834709

\end{thebibliography}
\end{document}